\documentclass[12pt]{amsart}
\usepackage{latexsym,amscd,amssymb,graphicx,epstopdf}

\theoremstyle{plain}
  \newtheorem{theorem}{Theorem}[section]
  \newtheorem{proposition}[theorem]{Proposition}
  \newtheorem{lemma}[theorem]{Lemma}
  \newtheorem{corollary}[theorem]{Corollary}
  
\theoremstyle{definition}
  \newtheorem{definition}[theorem]{Definition}
  \newtheorem{example}[theorem]{Example}

 \theoremstyle{remark}
  \newtheorem{remark}[theorem]{Remark}

\numberwithin{equation}{section}

\def\RR{{\mathbb R}}

\def\CC{{\mathbb C}}

\def\red{\mathrm{red}}

\begin{document}

\title[A poset map to Bruhat order]{Sorting orders, subword complexes, Bruhat order and total positivity}
\author{Drew Armstrong}
\author{Patricia Hersh}
\address{Department of Mathematics, University of Miami, 1365 Memorial Drive, Ungar 515, Coral Gables, FL 33146}
\email{armstrong@math.miami.edu}
\address{Box 8205, North Carolina State University, Raleigh, NC 27695-8205}
\email{plhersh@ncsu.edu}

\thanks{First author supported by NSF grant DMS-0603567}
\thanks{Second author was supported by NSF grant DMS-0757935}


\begin{abstract}
In this note we construct a poset map from a Boolean algebra to the Bruhat order
which unveils an interesting connection between subword complexes, sorting orders, and certain totally nonnegative spaces. This relationship gives a new proof of Bj\"orner and Wachs' result \cite{BW} that the proper part of Bruhat order is homotopy equivalent to the proper part of a Boolean algebra --- that is, to a sphere. We also obtain a geometric interpretation for sorting orders. We conclude with two new results: that the intersection of all sorting orders is the weak order, and the union of sorting orders is the Bruhat order.
\end{abstract}

\maketitle

\dedicatory{Dedicated to Dennis Stanton on the occasion of his 60th birthday.}

\section{Introduction and terminology}
\label{definition-section}

The sorting orders (cf. \cite{Ar}) and subword complexes (cf. \cite{KM}) are combinatorial structures defined in terms of a finite Coxeter group $W$. In this note we show how these structures both arise 
in the context of
a certain poset map from a Boolean algebra to the Bruhat order on $W$. A consequence of this poset map is a new proof of the homotopy equivalence of the proper part of any Bruhat interval to a
sphere. 

We also give the following refinement of Armstrong's result from \cite{Ar} that the sorting orders lie in between the weak order and Bruhat order on $W$. Given $w\in W$, let $[\hat{0},w]_R$ denote the weak order interval from $\hat{0}$ to $w$. We show that the intersection of all sorting orders for $w$ is the weak order on $[\hat{0},w]_R$ and that the union of all sorting orders for $w$ is the Bruhat order on $[\hat{0},w]_R$. These results apply to the whole group by taking $w$ to be the longest element 
$w_\circ$.

We will need the following concepts.

\subsection{Poset Topology and the Quillen Fiber Lemma}

Given a partially ordered set (poset) $(P,\leq)$, 
we say that $x\prec y$ is a {\sf cover relation} if we have $x<y$ and if there does not exist $z\in P$ such that $x<z<y$.
Recall that the {\sf order complex}, denoted $\Delta(P)$, 
of a finite poset $P$ is the simplicial complex whose $i$-faces are the chains $z_0<z_1<\cdots <z_i$ of comparable poset elements. Let $P_{\leq p}$ denote the subposet of elements $q\in P$ satisfying $q\leq p$.   Given any simplicial complex $K$, let $F(K)$
be its {\sf face poset}, i.e. the partial order on the faces of $K$ by containment.  More generally,
the {\sf closure poset} of a cell complex is the partial order on its cells given by $\sigma \le \tau $
if and only if $\sigma \subseteq \overline{\tau }$.

\begin{remark}\label{barycenter}
For any 
simplicial complex $K$, $\Delta ( F(K) \setminus \{ \emptyset \} )$ is the first barycentric 
subdivision of $K$, hence is homeomorphic to $K$.
\end{remark}

We recall from Quillen ~\cite{Qu} the following result:

\begin{lemma}\label{quillen-fiber}
If $f: P\rightarrow P'$ is a poset map whose fibers $f^{-1}(P'_{\leq p})$ are all contractible,
then the order complexes of $P$ and $P'$ are homotopy equivalent.
\end{lemma}

Recall that the {\sf dual poset} of $P$, denoted $P^*$, is defined by setting $p\leq q$ in $P^*$ whenever $p\geq q$ in $P$.

\begin{remark}\label{dual-qfl}
Since a poset and its dual have the same order complex, we may restate the Quillen Fiber Lemma by requiring contractibility of $f^{-1}(P'_{\geq p})$ for each $p\in P'$.
\end{remark}

Define the {\sf closed interval} 
$[u,v]$ of a poset $P$ as $\{ z\in P | u\le z \le v\} $ and the {\sf open interval} $(u,v)$ as
$\{ z\in P | u < z < v\} $.

\subsection{Sorting Orders}
We recall from \cite{Ar} the definition of the sorting orders. Let $W$ be a finite Coxeter group with simple generators $S=\{s_1,\ldots,s_n\}$. Let $Q$ be a word in the generators whose product is $w\in W$. If there is no such word with fewer symbols, we call $Q$ {\sf reduced}; the {\sf length} $\ell(w)$ is the number of symbols in a reduced word for $w$. We define the {\sf (right) weak order} on $W$ by setting $u\leq_R v$ if there exists a reduced word for $v$ such that $u$ occurs as a prefix. The {\sf Bruhat order} is defined by setting $u\leq_B v$ if there exists a reduced word for $v$ such that $u$ occurs as an arbitrary subword.

Now fix a reduced word $Q$  for some element $w\in W$. It turns out that the elements which occur as subwords of $Q$ are the elements of the Bruhat interval $[\hat{0},w]_B$. To each element $u\in [\hat{0},w]_B$ we associate a subset $Q(u)$ of the indices of $Q$ whose corresponding subword is reduced for $u$ and which is the lexicographically first subset with this property; such a subword exists by the definition of Bruhat order, i.e. we have $u\le_{Br} v$ if and only if any reduced word for $v$ has a subword that is a reduced word for $u$. The {\sf $Q$-sorting order} is defined by setting $u\leq_Q v$ whenever $Q(u)\subseteq Q(v)$. Armstrong \cite{Ar} proved the following.

\begin{theorem}[Ar, Theorem 4.2]\label{sorting}
Let $Q$ be a reduced word for $w\in W$. Given $u$ and $v$ in the Bruhat interval $[\hat{0},w]_B$, we have
\begin{equation*}
u\leq_R v \quad\Rightarrow\quad u\leq_Q v \quad\Rightarrow\quad u\leq_B v.
\end{equation*}
That is, the $Q$-sorting order is between the weak and Bruhat orders.
\end{theorem}

\subsection{The 0-Hecke Algebra}

Given a finite Coxeter group $W$ with generators $S=\{s_1,\ldots,s_n\}$, 
recall that the corresponding $0$-Hecke algebra has a 
generating set $X=\{x_1,\ldots,x_n\}$ with the 
following relations (up to a sign). Each 
Hecke generator is idempotent, namely satisfies
$x_i^2=x_i$, whereas the Coxeter generators are involutions $s_i^2=e$. 
Also, for each Coxeter braid relation $(s_is_j)^{m(i,j)}=e$ we get a corresponding
 Hecke braid relation of the form $x_ix_jx_i\cdots=x_jx_ix_j\cdots$, with $m(i,j)$ 
 alternating terms on each side.
Recall from ~\cite{BB}:

\begin{theorem}[BB, Theorem 3.3.1]\label{word-property}
Let $(W,S)$ be a Coxeter group generated by $S$.  Consider $w\in W$.  
\begin{enumerate}
\item 
Any expression $s_{i_1}s_{i_2}\cdots 
s_{i_d}$ for $w$ can be transformed into a reduced expression for $w$ by a sequence of
nil-moves and braid-moves.
\item
Every two reduced expressions for $w$ can be connected via a sequence of braid-moves.
\end{enumerate}
\end{theorem}

This implies that
an arbitrary word in the $0$-Hecke generators may be {\sf reduced} by a sequence of nil-moves $x_i^2\to x_i$ and braid moves $x_ix_jx_i\cdots \to x_jx_ix_j\cdots$. A reduced
word in the $0$-Hecke generators corresponds to a reduced
word in the Coxeter generators by switching occurrences of $x_i$ with $s_i$. Note that likewise
nonreduced  words in the 0-Hecke algebra correspond to nonreduced words in the associated
Coxeter group.  Our upcoming 
poset map from a Boolean algebra to Bruhat order
will take the subexpressions of a fixed reduced expression in the 
0-Hecke algebra and send them to Coxeter group elements.  

\subsection{Subword Complexes}

We recall from \cite{KM} and \cite{KM2} 
the definition of a subword complex. Let $Q$ be a (not necessarily reduced) 
word in the generators  and let $w\in W$ be an arbitrary element of the Coxeter group. The {\sf subword complex} $\Delta(Q,w)$ has as its ground set the indexing positions of the list $Q$ and its facets are the complementary sets to subwords of $Q$ that are reduced expressions for $w$. Since the set of faces of a simplicial complex is closed under taking subsets, the faces of $\Delta(Q,w)$ are precisely the complements of the 
subwords of $Q$ that contain a reduced expression for $w$.

Knutson and Miller \cite{KM} proved that subword complexes are vertex decomposable --- hence shellable --- and they gave a characterization of the homotopy type. To describe this we will need the notion of a {\sf Demazure product}. For a commutative ring $R$, define a free $R$-module with generators indexed by the Coxeter group $\{e_w : w\in W\}$ and with multiplication as follows. Let $w\in W$ be a group element and let $s\in S$ be a simple generator. Let $\ell:W\to\mathbb{Z}$ denote the word length on $W$ with respect to the generators $S$. If $s$ is a right descent of $w$ --- that is, if $\ell(w)>\ell(ws)$ --- we set $e_we_s=e_w$, and otherwise we set $e_we_s=e_{ws}$.

\begin{remark}
Notice that the Demazure product is equivalent to the multiplication rule for the 0-Hecke 
algebra from the previous section.
\end{remark}

\begin{theorem}[KM, Corollary 3.8]\label{subword}
The subword complex $\Delta (Q,w)$ is either a ball or a top-dimensional sphere.  It is a sphere if and only if the Demazure product applied to $Q$ yields exactly $w$. 
\end{theorem}

In the language of 
Section ~\ref{poset-map-section}, 
this means that we get a ball unless a certain poset map applied to $Q$ yields exactly the group element $w$, rather than yielding an element strictly above $w$ in Bruhat order.

Subword complexes arose in \cite{KM2} as the Stanley-Reisner 
complexes of  initial ideals  for the (determinantal) vanishing ideals of
matrix Schubert varieties.  Shellability of subword complexes provides one way
(of many) to explain why matrix Schubert varieties are Cohen-Macaulay, and the 
combinatorics of subword complexes provides insight into formulas for 
Grothendieck polynomials which are the $K$-polynomials of matrix Schubert
varieties.  

\subsection{Stratified Totally Nonnegative Spaces}\label{stratified-spaces}
Here, and in the following sections, we will use the notation $(i_1,i_2,\ldots,i_d)$ to denote either the word $s_{i_1}\cdots s_{i_d}$ in the Coxeter generators or the word $x_{i_1}\cdots x_{i_d}$ in the $0$-Hecke generators, depending on context.

Recall that a matrix is totally nonnegative if each minor is nonnegative.
It was shown by Whitney (\cite{Wh}) that every $n\times n$ totally nonnegative,
upper triangular matrix with $1$'s on the diagonal can be written as a product of 
matrices $\{ x_i(t): t\in \RR_{\ge 0},\, 1\le i \le n-1 \} $ where $x_i(t) = I_n + t E_{i,i+1}$
for $I_n$ the identity matrix and
$E_{i,i+1}$ the matrix whose only nonzero entry is a 1 which is in row $i$ and
column $i+1$.  Lusztig generalized this type A notion to all semisimple, simply 
connected algebraic groups over $\CC $ split over $\RR $.  
Lusztig showed for any reduced word $(i_1,\dots ,i_d)$ that the
set $Y_{s_{i_1}\cdots s_{i_d}}$
of points $\{ x_{i_1}(t_1)\cdots x_{i_d}(t_d): t_1,\dots ,t_d\in \RR_{>0} \} $ is 
homeomorphic to an open ball. Now consider the decomposition of the closure of
this space, namely $\{ x_{i_1}(t_1)\cdots x_{i_d}(t_d): t_1,\dots ,t_d \in \RR_{\ge 0} \} $, 
into cells defined by which subset of parameters are nonzero. Fomin and Shapiro \cite{FS} showed that the closure poset for this decomposition is the Bruhat order, and conjectured that this stratified space is a regular CW complex; this was recently proven by Hersh \cite{He} via a map on points of a space 
which induces a map on cells, namely the poset map of the next section.

\section{A Poset Map From Subwords to Bruhat Order}\label{poset-map-section}

Let us now  define a poset map from a Boolean algebra to Bruhat order which is
derived from
 Lusztig's map $(t_1,\dots ,t_d) \rightarrow x_{i_1}(t_1)\cdots x_{i_d}(t_d)$,
where $(i_1,\dots ,i_d)$ is any fixed reduced word.  First note that 
setting the parameter $t_j$  to 0 amounts to replacing the
matrix $x_{i_j}(t_j)$ by the identity matrix, essentially choosing the subword of
$(i_1,\dots ,i_d)$ with $i_j$ eliminated, hinting at a connection to the subword complexes
of \cite{KM}.   The relations
$x_i(t_1)x_i(t_2) = x_i(t_1+t_2)$ and 
$$x_i(t_1)x_{i+1}(t_2)x_i(t_3) = 
x_{i+1}(\frac{t_2t_3}{t_1+t_3})x_i (t_1+t_3)x_{i+1}(\frac{t_1t_2}{t_1+t_3})$$
along with similar relations corresponding to higher degree braid relations 
(cf. \cite{FZ})
allow us to use the relations in the 0-Hecke algebra to determine how to assign
points $(t_1,\dots ,t_d)\in \RR_{\ge 0}^d$ to cells in $\overline{Y_{s_{i_1}\cdots s_{i_d}}}$,
i.e. to elements of Bruhat order.

The poset map we now analyze
comes from choosing a subset of the parameters $\{ t_1,\dots ,t_d\} $ to
set to 0, taking the resulting subword of $(i_1,\dots ,i_d)$ and determining the Coxeter
group element which has this as one of its (not necessarily reduced) words in the
0-Hecke algebra.  That is:

\begin{definition}
Let $Q=(i_1,\dots ,i_d)$ 
be a reduced word for a Coxeter group element $w\in W$. Define a map $f$ from the Boolean algebra $2^Q$ to the Bruhat interval $[\hat{0},w]_B$ as follows. Given a subset
$\{ j_1,\dots ,j_r \} $ of the indices $\{ 1,\dots ,d\} $ 
of $Q$, map this to a (possibly non-reduced) expression $x_{i_{j_1}}\cdots x_{i_{j_r}} $
in the $0$-Hecke generators; use braid moves and nil-moves to obtain a reduced $0$-Hecke 
expression; then map this to a Coxeter group element by replacing each $x_k$ with the 
corresponding Coxeter group generator $x_k$.
\end{definition}

For example, consider the Coxeter group of type $A_3$ with generators $\{s_1,s_2,s_3\}$, and consider the reduced word $Q=(1,2,3,1,2,1)$ for the longest element $w_\circ=s_1s_2s_3s_1s_2s_1$. The index set $\{1,2,4,5\}\in 2^Q$ gets mapped to the $0$-Hecke word $x_1x_2x_1x_2$, which reduces to
\begin{equation*}
x_1(x_2x_1x_2)\to x_1(x_1x_2x_1)\to (x_1x_1)x_2x_1\to x_1x_2x_1.
\end{equation*}
Finaly, this maps to the group element $s_1s_2s_1$.

We have the following results.

\begin{proposition}
Each fiber of $f$ contains exactly one element of 
the $Q$-sorting order, and every element of the $Q$-sorting order belongs to some fiber of $f$.  The subposet of the Boolean algebra $2^Q$ comprised of these elements is isomorphic to the $Q$-sorting order.
\end{proposition}

\begin{proof}
This is immediate from the definition of the sorting order.
\end{proof}

\begin{proposition}\label{fibers}
For $u\in [\hat{0},w]_B$, the fiber $f^{-1}([u,w]_B)$ is dual to the face poset of the subword complex $\Delta(Q,u)$.  The fiber $f^{-1}((u,w)_B)$ is dual to the face poset for
$\partial (\Delta (Q,u))\setminus \{ \emptyset \} $, i.e. for the boundary of the subword
complex $\Delta(Q,u)$ with the empy set removed.
\end{proposition}

\begin{proof}
For $u\in[\hat{0},w]_B$, note that the elements $S=\{j_1,j_2,\ldots,j_r\}\in 2^Q$ with $f(S)=u$ are exactly those such that $x_{j_1}x_{j_2}\cdots x_{j_r}$ reduces via braid moves and nil-moves to a reduced expression $x_{i_1}x_{i_2}\cdots x_{i_t}$ such that $s_{i_1}s_{i_2}\cdots s_{i_t}=u$. In particular, the minimal such $S$ are precisely the subwords of $Q$ which are reduced words for $u$. Thus the fiber $f^{-1}(u)$ consists of all subwords of $Q$ which contain a reduced expression for $u$.
This is exactly the dual of the subword complex $\Delta(Q,u)$.  

To see the latter claim, notice that a face $F$
in $f^{-1}([u,w]_B)$ has the same Demazure product as $u$  if and only if $F$ is in the 
interior of $\Delta (Q,u)$, because we will show that every face $G$ of $\Delta (Q,u)$
which contains $F$ and has dimension exactly one less than the facets of $\Delta (Q,u)$ is
itself in the interior of $\Delta (Q,u)$.  This follows from the fact that any such $G$ must have
the same Demazure product as  $u$ while involving exactly one more letter
than a reduced expression for $u$, implying $G$ is contained in exactly two facets of
$\Delta (Q,u)$, hence in its interior.
\end{proof}

Hence we obtain a new proof characterizing the homotopy type of Bruhat order (this result was originally proved by Bj\"orner and Wachs \cite{BW}, who invented the technique of CL-shellability for this purpose).

\begin{corollary}
The order complex for the open Bruhat interval $(\hat{0},w)_B$ 
is homotopy equivalent to the order complex for the 
proper part of the Boolean algebra $2^Q$ of dimension $\ell(w)$, hence to a sphere
$S^{\ell(w) -2}$.  More generally, the order complex for the open 
Bruhat interval $(u,w)_B$ is homotopy equivalent to a sphere $S^{\ell(w)-\ell(u)-2}$.
\end{corollary}

\begin{proof} 
This is a simple application of the Quillen Fiber Lemma, namely 
Lemma ~\ref{quillen-fiber}, which originally appeared in [Qu].  Proposition ~\ref{fibers}
shows that the fibers of our poset map are face posets of subword complexes. 
It was shown
in ~\cite{KM} that subword complexes are shellable.  By Theorem  ~\ref{subword}, we 
see  in particular that all fibers for the proper part of the interval $[\hat{0},w]_B$
are contractible, since the fact that the maximal element (that is, $Q$) is a 
reduced word means that any proper subset has distinct Demazure product from it.  
Thus, by Remark ~\ref{barycenter} the order complexes of the
fibers in this case are homeomorphic to contractible subword complexes.
This gives the requisite contractibility of fibers,  completing our proof of the 
first claim.

The Quillen Fiber Lemma also tells us for any open Bruhat interval $(u,w)_B$ that its
order complex is homotopy equivalent to the order complex of the
subposet $f^{-1}(u,w)$ of a Boolean algebra.  By Proposition ~\ref{fibers}, this preimage
is the boundary of a ball, hence is a sphere.  Its dimension is immediate from the definition
of subword complexes.
\end{proof}

In ~\cite{He}, Hersh proved that the stratified spaces from 
Section ~\ref{stratified-spaces} are regular CW complexes with 
Bruhat order as their closure posets.  Her method was to start with a simplex, namely 
$$\{ (t_1,\dots ,t_d) \in \RR_{\ge 0}^d : t_1 + \cdots + t_d = 1 \} ,$$
having the Boolean algebra
as its closure poset, and then perform a series of collapses preserving 
homeomorphism type and regularity.   See ~\cite{He} for details, including the precise
notion of  ``collapse'' being used, as this is quite involved.  Example ~\ref{collapse-example}
is included in hopes of conveying in a brief manner some intuition for these collapses.

What is relevant here is (1) that the Bruhat order is the closure poset for the regular CW
complex obtained at the end of this collapsing
process, and (2) that our poset map describes which faces
of the simplex are mapped to a particular cell in the regular CW complex having  
Bruhat order as its closure poset.

\begin{example}\label{collapse-example}
Given the reduced word $(1,2,1)$ in type $A_2$, \cite{He} performs a single collapse on the simplex $\{ (t_1,t_2,t_3)\in \RR_{\ge 0}^3: t_1 + t_2 + t_3 = 1\} $,  eliminating the
open cell $\{ (t_1, 0 , t_3 )\in \RR_{\ge 0}^3 : t_1 + t_3 = 1\} $ by identifying points
which have $t_1 + t_3 = k$ for any constant $k$; in this case we just have $k=1$, but 
in larger examples we have a family of constants $k$ ranging from $0$ 
to $1$.  In effect, this moves the cell $\{ (t_1,0,0) \in \RR_{\ge 0}^3  : t_1 = 1 \} $ 
across the cell $\{ (t_1,0,t_3)  \in \RR_{\ge 0}^3 : t_1 + t_3 = 1\} $ being collapsed,
thereby identifying the cell $\{ (t_1,0,0)\in \RR_{\ge 0}^3 : t_1 = 1\} $ with the cell
$\{ (0,0,t_3)\in \RR_{\ge 0}^3 : t_3 = 1\} $, reflecting the fact that $s_1 \cdot 1 \cdot 1$ and $1\cdot 1\cdot s_1 $ are both reduced words for the same Coxeter group element.
\end{example}

Some cell incidences are present from the beginning to the end of this collapsing process
appearing in ~\cite{He}, whereas others are created along the way.
The sorting order captures this distinction:

\begin{proposition}\label{geometric-meaning}
The sorting order on a finite Coxeter group is generated by precisely
those covering relations on Bruhat order between cells that were already incident in
the simplex, i.e. it leaves out exactly
those cell incidences that were introduced by collapses.
\end{proposition}

\begin{proof}
Consider a subword $P=(i_{j_1},\dots ,i_{j_k})\in 2^Q$ of the given word $Q=(i_1,\dots ,i_d)$ such that $\{ j_1,\dots ,j_k \} $ is not a $Q$-sorting word --- that is, it is not lexicographically first among subwords that give rise to the group element $f(P)$. Following the language of \cite{He}, this means that 
$P$ must include the larger element of some deletion pair in a 
word $P\vee P'$ obtained from $P$ by adding exactly one letter to obtain a subword 
of $Q$.  This means that the larger element of the deletion pair may be 
exchanged for the smaller one to obtain a lexicographically smaller reduced word for
the group element $f(P)$.  In particular, the face given by $P$
gets identified with a lexicographically earlier choice $P'$ by a collapsing step in \cite{He}
which collapses a higher dimensional cell $P\vee P'$ including both elements of the deletion pair
by moving $P$ across $F$ so as to identify it with $P'$.

Thus, the  faces given by the lexicographically first reduced 
expressions are essentially 
held fixed (ignoring unimportant stretching homeomorphisms on them) 
by the collapsing process of \cite{He} while all other faces are
collapsed onto these, creating new cell incidences in the process.
The incidences among the lexicographically first reduced expressions
are inclusions, which is precisely the definition of the $Q$-sorting order.  Examining 
the collapsing process in \cite{He}, we see that these are exactly 
the incidences that are present from the beginning to the end of the collapsing process --- that is, left unchanged by it. All other incidence relations
in Bruhat order are introduced by collapses.
\end{proof}

\begin{remark}
Proposition ~\ref{geometric-meaning} gives a way to think about the $Q$-sorting order as a closure poset, and also a geometric explanation for the fact that its cover relations are a subset of Bruhat covers. \end{remark}

\begin{corollary}
The covering relations in the $Q$-sorting order are a proper subset of
the covering relations in the Bruhat interval $[\hat{0},w]_B$.
\end{corollary}

Finally, we establish two new properties of the sorting orders.

\section{Unions and Intersections of Sorting Orders}
Given an element $w\in W$ of a finite Coxeter group, let $\red(w)$ denote the set of reduced words for $w$ in the generators $S$. For each $Q\in\red(w)$ there is a corresponding $Q$-sorting order on the elements of the Bruhat interval $[\hat{0},w]_B$. In this section, however, we will consider the $Q$-sorting order as restricted to the elements of the interval $[\hat{0},w]_R$ in (right) {\em weak} order. We will show that the intersection of $Q$-sorting orders, over $Q\in\red(w)$, is the weak order on $[\hat{0},w]_R$ and the union of these $Q$-sorting orders is the Bruhat order on $[\hat{0},w]_R$. Our results apply to the entire group $W$ by taking the longest element $w=w_\circ$.

\begin{theorem}
The intersection over $Q\in\red(w)$ of $Q$-sorting orders is the weak order on the weak interval $[\hat{0},w]_R$.
\end{theorem}

\begin{proof}
Since weak order is contained in each $Q$-sorting order (see Theorem \ref{sorting}), it is contained in the intersection. Recall also that each $Q$-sorting cover $u\prec_Q v$ is a Bruhat cover $u\prec_B v$ (Theorem \ref{sorting}). Thus, we must show: Given $u$ and $v$ in $[\hat{0},w]_R$ such that $u\prec_B v$ and $u\not\leq_R v$, there exists a reduced word $Q\in\red(w)$ such that $u\not\leq_Q v$.

Let $Q'\in\red(u)$ be any reduced word for $u$. Since we have $u\leq_R w$ by assumption, there exists a reduced word $Q\in\red(w)$ for $w$ which contains $Q'$ as a prefix. We claim that $u\not\leq_Q v$ in $Q$-sorting order. Indeed, the $Q$-sorting subword $Q(u)\subseteq Q$ corresponding to $u$ is just the prefix $Q'\subseteq Q$. Let $Q(v)\subseteq Q$ be the $Q$-sorting word for $v$; that is, the lex-first reduced word for $v$ as a subword of $Q$. Suppose now that $u\leq_Q v$; i.e., that $Q(u)\subseteq Q(v)$. Then $Q(v)$ is a reduced word for $v$ containing the reduced word $Q(u)$ for $u$ as a prefix, which implies that $u\leq_R v$. This is a contradiction.
\end{proof}

\begin{theorem}
The union over $Q\in\red(w)$ of $Q$-sorting orders is the Bruhat order on the weak interval $[\hat{0},w]_R$.
\end{theorem}

\begin{proof}
Given $u$ and $v$ in $[\hat{0},w]_R$ such that $u\prec_B v$ is a Bruhat cover, we must find a reduced word $Q\in\red(w)$ such that $u\leq_Q v$.

Since $v\leq_R w$ in weak order, there exists a reduced word for $v$ which is a prefix of $Q$. Since a prefix is clearly lex-first, this prefix is the $Q$-sorting word $Q(v)\subseteq Q$. Suppose that $v$ has length $k$, so that $Q(v)$ corresponds to positions $\{1,2,\ldots,k\}$. By the subword definition of Bruhat order above, there exists a reduced word $Q'$ for $u$ in positions $\{1,2,\ldots,k\}-\{i\}$ for some $1\leq i\leq k$. We claim that $u\leq_Q v$. Indeed, we already have $Q'\subseteq Q(v)$. The $Q$-sorting word $Q(u)$ for $u$ is a reduced word for $u$ which precedes $Q'$ lexicographically. Since $Q(v)$ is a prefix of $Q$ we must have $Q(u)\subseteq Q(v)$, or $u\leq_Q v$.
\end{proof}

For example, consider the Coxeter group $W$ of type $B_2$ with generators $\{1,2\}$ and longest element $w_\circ=1212=2121$. In this case we have
\begin{equation*}
\red(w_\circ)=\left\{ Q_1=(1,2,1,2), Q_2=(2,1,2,1)\right\}.
\end{equation*}
Figure \ref{fig:b2} displays the weak order, $Q_1$-sorting order, $Q_2$-sorting order, and Bruhat order on the full group. Here we have used zeroes as placeholders, to indicate the indices not in a given word. Note that the weak order is the intersection and the Bruhat order is the union of the two sorting orders. Next, observe that there is only one reduced word for the element $w=212$, namely $Q=(2,1,2)$. Note that the $Q$-sorting order is just the interval $[\hat{0},212]_{Q_2}$ in $Q_2$-sorting order. Note also that this order coincides {\em neither} with the weak {\em nor} the Bruhat order on the set $[\hat{0},212]_Q=[\hat{0},212]_B$. Our above results instead tell us that the weak order, $Q$-sorting order, and Bruhat order all coincide on the {\em weak} interval $[\hat{0},212]_R$.

\begin{figure}
\begin{center}
\includegraphics[scale=.75]{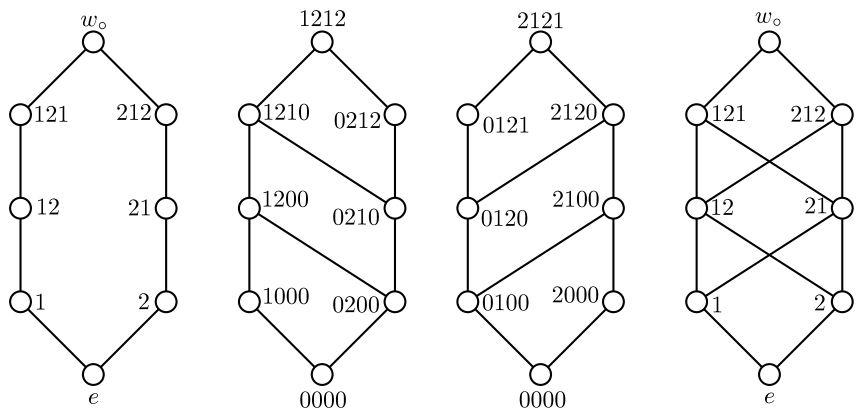}
\end{center}
\caption{The weak order, $Q_1$-sorting order, $Q_2$-sorting order, and Bruhat order}
\label{fig:b2}
\end{figure}

\section{acknowledgments}

The authors thank Ezra Miller for helpful comments.


\begin{thebibliography}{70}


\bibitem[Ar]{Ar} 
D. Armstrong, The sorting order on a Coxeter group, to appear in Journal of Combin.
Theory, Ser. A

\bibitem[Bj]{Bj}
A. Bj\"orner, Posets, Regular CW complexes and Bruhat order, European J. Combin.
{\bf 5} (1984), no. 1, 7--16.

\bibitem[BB]{BB}
A. Bj\"orner and F. Brenti, Combinatorics of Coxeter groups, Graduate Texts in 
Mathematics {\bf 231}, Springer, 2005.

\bibitem[BW]{BW}
A. Bj\"orner and M. Wachs, Bruhat order of Coxeter groups and
shellability, Adv.  Math. {\bf 43} (1982), no. 1, 87--100.




\bibitem[FS]{FS}
S. Fomin and M. Shapiro, Stratified spaces formed by totally positive varieties.
Dedicated to William Fulton on the occasion of his 60th birthday.
Michigan Math. J. {\bf 48} (2000), 253--270.

 
\bibitem[FZ]{FZ}
S. Fomin and A. Zelevinsky, Double Bruhat cells and total positivity, Journal
Amer. Math. Soc., {\bf 12} (April 1999), no. 2, 335--380.


\bibitem[He]{He}
P. Hersh, Regular cell complexes in total positivity, submitted.

\bibitem[Hu]{Hu}
J. Humphreys, Reflection groups and Coxeter groups, Cambridge studies in advanced
mathematics {\bf 29}, Cambridge University Press, 1990.


\bibitem[KM]{KM}
A. Knutson and E. Miller, Subword complexes in Coxeter groups, Adv. Math., {\bf 184}, 
(2004), no. 1, 161--176.
 
 \bibitem[KM2]{KM2}
 A. Knutson and E. Miller, Gr\"obner geometry of Schubert polynomials, Annals of Math., {\bf 161},
 (2005), 1245--1318.
 




\bibitem[Qu]{Qu}
D. Quillen, Homotopy properties of the poset of nontrivial $p$-subgroups of a group,
Advances in Math. {\bf 28} (1978), 101--128.


\bibitem[Wh]{Wh}
A. M. Whitney, A reduction theorem for totally positive matrices, J. d'Analyse Math. {\bf 2}
(1952), 88--92.




\end{thebibliography}
\end{document}